\documentclass[11pt]{article}

\usepackage[utf8]{inputenc}
\usepackage[margin=1in]{geometry}
\usepackage{amsmath,amsthm,amsfonts}
\usepackage{bbm}
\usepackage{cite}
\usepackage{color}
\usepackage{graphicx} 
\usepackage{amsmath} 
\usepackage{amsxtra} 
\usepackage{amssymb} 
\usepackage{amsthm} 
\usepackage{latexsym} 
\usepackage{setspace} 
\usepackage[margin=1in]{geometry} 
\usepackage[titles]{tocloft} 
\usepackage{comment}
\usepackage{tikz}
\usepackage{url}

\newtheorem{theorem}{Theorem}[section]
\newtheorem{lem}[theorem]{Lemma}

\theoremstyle{definition}
\newtheorem{asm}{Assumption}

\theoremstyle{remark}
\newtheorem{exmp}[theorem]{Example}

\DeclareMathOperator{\Adm}{Adm}
\DeclareMathOperator{\Mod}{Mod}
\newcommand{\ones}{\mathbf{1}}
\newcommand{\trp}{\text{trp}}
\newcommand{\rv}[1]{\underline{#1}}

\title{Modulus of time-respecting paths}
\author{Nathan Albin and Vikenty Mikheev}

\begin{document}

\maketitle
\begin{abstract}
On a static graph, the $p$-modulus of a family of paths reflects both the lengths of these paths as well as their diversity; a family of many short, disjoint paths has larger modulus than a family of a few long overlapping paths.  In this work, we define a version of $p$-modulus for time-respecting paths on temporal graphs.  This formulation makes use of a time penalty function as a means of discounting paths that take a relatively long time to traverse, thus allowing modulus to capture temporal information about the family as well. By means of a transformation, we show that this temporal $p$-modulus can be recognized as a $p$-modulus problem on a static graph and, therefore, that much of the known theory of $p$-modulus of families of objects can be translated to the case of temporal paths. We demonstrate some properties of temporal modulus on examples. 

\end{abstract}



\section{Introduction}
Discrete $p$-modulus (see, e.g.,~\cite{albin2015modulus,ProbInterp2016,Blocking2018,APDM2016}) has proven to be a powerful and versatile tool for exploring the structure of graphs and networks.  Applications discovered thus far include clustering and community detection~\cite{Cluster2017}, the construction of a large class of graph metrics~\cite{ModMetric}, measures of centrality~\cite{smapcs}, hierarchical graph decomposition~\cite{achppcst}, and the solution to game-theoretic models of secure network broadcast~\cite{albin2019spanning}.  The flexibility of $p$-modulus allows it to be applied to a wide variety of graph types; the essential definitions are easily adapted for directed or undirected graphs, for weighted or unweighted graphs, for simple graphs or multigraphs, or even for hypergraphs.  In this paper, we begin the exploration of $p$-modulus applied to temporal networks---that is, networks whose structure can change over time.

Roughly speaking, $p$-modulus provides quantitative information about a family of objects on a graph.  The concept of object is very flexible, encompassing such concepts as walks, paths, cuts, spanning trees, partitions and flows.  By applying the $p$-modulus framework to different families of objects, one can gain interesting insights into the structure of the underlying graph.  It is expected that, similarly, the application of $p$-modulus to various families of objects on a temporal graph will provide analogous insight.  As a beginning step in this direction, we develop a theory of $p$-modulus for a family of time-respecting paths.

The key contributions of this paper are as follows.
\begin{itemize}
\item We extend the definition of $p$-modulus to families of time-respecting paths on temporal graphs through the introduction of a temporal penalty function.  (Section~\ref{sec:obj_fam_usage})
\item We demonstrate that, with this definition, many important theorems concerning static graphs have direct counterparts on temporal graphs.  (Theorems~\ref{thm:trp-monotone},~\ref{thm:temp-prob-interp} and~\ref{thm:temp-sigma})
\item We establish a connection between the new temporal modulus and the static modulus as the temporal penalty vanishes. (Theorem~\ref{thm:temp-mod-limit})
\item We show, through a series of examples, an idea of what information $p$-modulus is able to capture on a temporal graph. (Section~\ref{sec:examples})
\end{itemize}

The remainder of this paper is organized as follows.  In Section~\ref{sec:mod-paths}, we summarize the definitions and properties of $p$-modulus, especially in the context of path families.  In Section~\ref{sec:mod-trp}, we present a technique for defining $p$-modulus on families of time-respecting paths explore some of the properties of the new temporal modulus.  In Section~\ref{sec:examples}, we present a series of examples establish some heuristics about temporal modulus.

\section{Modulus on static graphs}
\label{sec:mod-paths}

Here, we summarize the relevant definitions and theorems for modulus on static graphs (see, e.g.,~\cite{Blocking2018}). Consider a {\it finite weighted (static) graph} $G=(V,E,\sigma)$ with vertex set $V$, edge set $E$, and positive edge weights $\sigma:E\to(0,\infty)$.  For the definitions that follow, further details, such as whether the graph is directed or undirected, simple or a multigraph, etc., are not particularly important.  It should be understood that these definitions apply to all these types of graph.  The case of unit edge weights, $\sigma\equiv 1$, can be interpreted as the case of an unweighted graph.  In what follows, it will frequently be convenient to adopt linear algebra notation.  To this end, we shall also think of $\sigma$ as a positive abstract vector indexed by the edge set $E$, and write $\sigma\in\mathbb{R}^E_{>0}$.  In order to distinguish this case from that of temporal graphs discussed in the next section, we will sometimes refer to $G=(V,E,\sigma)$ as a \emph{static graph}.

\subsection{Objects, families, usage}
\label{sec:obj_fam_usage}
By a \emph{family of objects} on $G$, we mean a collection, $\Gamma$, of abstract objects defined through some relationship with $G$.  For example, each of the following is a possible choice for $\Gamma$.
\begin{itemize}
    \item The family of spanning trees of $G$.
    \item The family of simple cycles in $G$.
    \item The family of all paths connecting two distinct vertices $s$ and $t$ in $G$.
    \item The family of all walks that traverse at least $k$ edges (where $k$ is some positive integer).
\end{itemize}
In order to keep the narrative simple, we shall assume that $\Gamma$ is a finite set.  A theory can also be established for infinite families, but more care must be taken (see, e.g.,~\cite[Sec.~7]{APDM2016}).

A \emph{usage matrix}, $\mathcal{N}\in\mathbb{R}^{\Gamma\times E}_{\ge 0}$, is an abstract matrix assigning a nonnegative value $\mathcal{N}(\gamma,e)$ to each pair $(\gamma,e)\in\Gamma\times E$.  There is no prescribed rule for making this definition, though the ``natural'' definition often suffices in most cases.  For example, if the objects in $\Gamma$ can be identified with subsets of edges (as in the case of spanning trees, cycles and paths), then a natural definition for $\mathcal{N}$ is
\begin{equation}
    \label{eq:natural-N}
    \mathcal{N}(\gamma,e) =
    \begin{cases}
    1 & \text{if }e\in\gamma,\\
    0 & \text{if }e\notin\gamma.
    \end{cases}
\end{equation}
For walks, it is often useful to store a traversal count instead.  For some sets of objects, it also proves useful to allow $\mathcal{N}$ to take non-integer values.

\subsection{Densities, admissibility and modulus}

By a \emph{density} on $G$, we mean a non-negative vector on the edges, $\rho\in\mathbb{R}^E_{\ge 0}$.  Each density provides a way to assess a type of cost on the objects in $\Gamma$.  It is helpful to think of $\rho(e)$ as a cost per unit usage for edge $e$.  The \emph{total usage costs} incurred by an object $\gamma\in\Gamma$, often referred to as the \emph{$\rho$-length} of $\gamma$, is defined as
\begin{equation*}
    \ell_\rho(\gamma) := \sum_{e\in E}\mathcal{N}(\gamma,e)\rho(e) = (\mathcal{N}\rho)(\gamma).
\end{equation*}
In words, the $\rho$-length of an object is the accumulation of its edge usage costs.

In the theory of $p$-modulus, special attention is paid to the densities that assess a minimum unit cost to all objects in the family of interest.  Such densities are called \emph{admissible} for the family $\Gamma$.  The collection of all admissible densities, $\Adm(\Gamma)$, is defined as
\begin{equation*}
    \Adm(\Gamma) := \{\rho\in\mathbb{R}^E_{\ge0} : \ell_\rho(\gamma)\ge 1\;\forall\gamma\in\Gamma\}.
\end{equation*}
We will often express this set of lengths inequalities in linear algebra notation as $\mathcal{N}\rho\ge\ones$, where $\ones\in\mathbb{R}^{\Gamma}$ is the vector of all ones. To each density $\rho$ is assigned a \emph{$p$-energy} (parametrized by $p\in[1,\infty]$ as well as the edge weights $\sigma$), which is defined as
\begin{equation}\label{eq:energy}
\mathcal{E}_{p,\sigma}(\rho) 
= \begin{cases} \sum \limits_{e \in E} \sigma(e) \rho(e)^p & \mbox{if } 1\le p < \infty, \\ \sup \limits_{e \in E} \sigma(e) \rho(e) & \mbox{if } p=\infty. \end{cases}
\end{equation}
Finally, the \emph{$p$-modulus} of the family $\Gamma$ is defined as the value of an optimization problem:
\begin{equation}
\label{eq:DefMod}
\Mod_{p,\sigma}(\Gamma) = 
\min_{\rho\in\Adm(\Gamma)}
\mathcal{E}_{p,\sigma}(\rho).
\end{equation}

Geometrically, $p$-modulus is simply the distance in $\mathbb{R}^E$ (measured by a weighted $p$-norm) between the origin and the closed, convex set $\Adm(\Gamma)$.  From this, it is clear that the minimum is attained.  A density $\rho^*\in\Adm(\Gamma)$ whose energy coincides with the $p$-modulus is called an \emph{optimal density}.  When $p\in(1,\infty)$, the strict convexity of the $p$-norm implies that there is a unique optimal $\rho^*$.

\subsection{An example}\label{sec:mod-example}

The following is a helpful example found, e.g., in~\cite{albin2015modulus}.  Consider a graph composed of $k$ parallel paths, each with $\ell$ hops, connecting a vertex $s$ to another vertex $t$.  For simplicity, let $\sigma\equiv 1$.  For the family $\Gamma$, we choose $\Gamma=\Gamma(s,t)$, the family of paths connecting from $s$ to $t$ (so $|\Gamma|=k$).  From symmetry arguments, it is easy to guess that an optimal density is $\rho^*=\frac{1}{\ell}$, regardless the value of $p$ and, therefore, that
\begin{equation*}
    \Mod_{p,1}(\Gamma(s,t)) = \mathcal{E}_{p,1}(\rho^*) = k\ell\left(\frac{1}{\ell}\right)^p = \frac{k}{\ell^{p-1}}.
\end{equation*}
This demonstrates a general heuristic that modulus is made large by a combination of ``diversity'' (in this case, more paths from $s$ to $t$) and ``shortness'' (in this case, shorter paths from $s$ to $t$).  The parameter $p$ adjusts the relative sensitivity of modulus to these two quantities: smaller $p$ for more sensitivity to diversity, larger $p$ for more sensitivity to shortness.

\subsection{Interpretations of modulus}

In practice, both the value $\Mod_{p,\sigma}(\Gamma)$ as well as the value of the optimal density $\rho^*$ contain useful information about the underlying graph $G$.  Some intuition about the meaning of modulus can be developed by considering specific examples.  In the next section, we will use the intuition gathered here to propose method for adapting $p$-modulus to families of objects on temporal graphs.

\subsubsection{Modulus as dissipated power}

For $p\in(1,\infty)$, it is possible to relate $p$-modulus to the dissipated power in a (generally nonlinear) resistor network.  In this interpretation, we view the edges of the graph as nonlinear resistors and the vertices as junctions at which these resistors are joined.  Along each resistor, we assume a relationship between the potential drop and current flow of the form
\begin{equation*}
    \text{current} = \sigma\times\text{voltage}^{p-1}.
\end{equation*}
(In the case $p=2$, this is simply Ohm's law for a resistor with conductance $\sigma$.)

Now consider distinct vertices $s,t\in V$ and suppose a unit voltage potential drop is induced between these two junctions.  From Dirichlet's minimum power principle, one would expect a steady-state response determined by a potential $\phi\in\mathbb{R}^V$ (the voltage at each junction) which minimizes the total dissipated power, leading to the following optimization problem.
\begin{equation}
\label{eq:nonlinear-power}
\begin{split}
    \underset{\phi\in\mathbb{R}^V}{\text{minimize}}\quad&\sum_{\{x,y\}\in E}\sigma(\{x,y\})|\phi(x)-\phi(y)|^p\\
    \text{subject to}\quad&
    \phi(s) = 0,\;\phi(t) = 1.
\end{split}
\end{equation}
Notice that, for each edge $e=\{x,y\}\in E$, we are associating $\sigma(e)$ with the conductance of the resistor assigned to that edge.
The relationship to $p$-modulus is expressed in the following Theorem~\cite[Thm.~4.2]{albin2015modulus}.
\begin{theorem}
Let $p\in(1,\infty)$, let $s,t\in V$ be distinct vertices, and let $\Gamma=\Gamma(s,t)$ be the family of paths connecting $s$ to $t$ in and undirected graph $G$, with $\mathcal{N}$ defined as in~\eqref{eq:natural-N}.  Let $\rho^*$ be the unique optimal density for $p$-modulus, and let $\phi^*$ be the unique optimal potential for~\eqref{eq:nonlinear-power}.  Then
\begin{equation*}
    \rho^*(e) = |\phi^*(x)-\phi^*(y)|\quad
    \text{for all }e=\{x,y\}\in E.
\end{equation*}
Consequently,
\begin{equation*}
    \Mod_{p,\sigma}(\Gamma)
    = \mathcal{E}_{p,\sigma}(\rho^*)
    = \sum_{\{x,y\}\in E}
    \sigma(\{x,y\})|\phi^*(x)-\phi^*(y)|^p.
\end{equation*}
\end{theorem}
In words, the $p$-modulus for the \emph{family of connecting paths} $\Gamma=\Gamma(s,t)$ coincides with the minimum dissipated power of the nonlinear resistor network.  Moreover, the optimal density $\rho^*(e)$ coincides with the absolute potential difference across the resistor associated with edge $e$.

\subsubsection{Modulus as a metric}

A related way of understanding modulus is through its connection to graph metrics.  As described in~\cite{Blocking2018}, if $G$ is an undirected graph then, for any $p\in(1,\infty)$ with conjugate H\"older exponent $q=p/(p-1)$, the function $\delta_p:V\times V\to\mathbb{R}$ defined as
\begin{equation*}
    \delta_p(x,y) = 
    \begin{cases}
    0 & \text{if }x = y,\\
    \Mod_{p,\sigma}(\Gamma(x,y))^{-\frac{q}{p}}
    & \text{if }x \ne y
    \end{cases}
\end{equation*}
is a metric on the vertices.  In the special case $p=q=2$, this metric coincides with the effective resistance metric due to the connection established above.  Moreover, it was shown that when $p=1$, $\Mod_{1,\sigma}(\Gamma(x,y))^{-1}$ coincides with the minimum cut metric (with edge capacities provided by $\sigma$) and when $p=\infty$, $\Mod_{\infty,\sigma}(\Gamma(x,y))^{-1}$ coincides with the shortest-path metric with edge lengths provided by $\sigma^{-1}$.

\subsubsection{The probabilistic interpretation of modulus and expected overlap}
\label{sec:prob-interp}

Finally, we review briefly the probabilistic interpretation of $p$-modulus developed in~\cite{ProbInterp2016,Blocking2018}.  Adopting the notation of the references, we fix a family $\Gamma$ and consider a random object $\rv{\gamma}\in\Gamma$.  (The underline notation is used to distinguish the random object $\rv{\gamma}$ from its possible instantiations $\gamma$.)  This random object is modeled through its distribution, or law, which takes the form of a probability mass function (pmf) $\mu:\Gamma\to\mathbb{R}$ that assigns to each object in $\Gamma$ its probability of selection.  That is, we think of $\mu$ as the probability
\begin{equation*}
    \mu(\gamma) = \mathbb{P}_\mu(\underline{\gamma}=\gamma),
\end{equation*}
with the subscripted $\mu$ indicating that the probability depends on the choice of distribution.  We denote by $\mathcal{P}(\Gamma)$ the set of all possible pmfs on $\Gamma$.  That is,
\begin{equation*}
    \mathcal{P}(\Gamma) =
    \left\{\mu\in\mathbb{R}^{\Gamma}_{\ge 0}
    : \mu^T\ones=1\right\}.
\end{equation*}
We shall use the notation $\underline{\gamma}\sim\mu$ to indicate the relationship between a random variable and its underlying pmf.

For a random object $\underline{\gamma}\sim\mu$ and a fixed edge $e\in E$, the quantity $\mathcal{N}(\underline{\gamma},e)$ is a real-valued random variable.  Its expectation, called the \emph{expected edge usage} for edge $e$ plays a special role in this interpretation of modulus, and is identified by the symbol $\eta(e)$.  That is,
\begin{equation*}
\eta(e) = \mathbb{E}_\mu(\mathcal{N}(\rv{\gamma},e))
= \sum_{\gamma\in\Gamma}\mathcal{N}(\gamma,e)\mu(\gamma) = (N^T\mu)(e).
\end{equation*}
In the special case that $\mathcal{N}$ takes the form~\eqref{eq:natural-N}, $\eta(e)$ can be understood as a marginal probability:
\begin{equation*}
\eta(e) = \sum_{\gamma\in\Gamma}\mathbbm{1}_{e\in\gamma}\mu(\gamma) = \mathbb{P}_\mu(e\in\rv{\gamma}).
\end{equation*}

The probabilistic interpretation of modulus developed in~\cite{ProbInterp2016,Blocking2018} is summarized in the following theorem.
\begin{theorem}\label{thm:prob-interp}
Let $\Gamma$ be a family of objects on $G$ and let $p\in(1,\infty)$ and $q=p/(p-1)$.  Define the conjugate edge weights $\hat{\sigma}(e)=\sigma(e)^{-\frac{q}{p}}$.  Then,
\begin{equation*}
\Mod_{p,\sigma}(\Gamma) = 
\left(\min_{\substack{\mu\in\mathcal{P}(\Gamma)\\\eta=\mathcal{N}^T\mu}}
\sum_{e\in E}\hat{\sigma}(e)\eta(e)^q\right)^{-\frac{p}{q}}
=
\left(\min_{\substack{\mu\in\mathcal{P}(\Gamma)\\\eta=\mathcal{N}^T\mu}}
\mathcal{E}_{q,\hat{\sigma}}(\eta)\right)^{-\frac{p}{q}}.
\end{equation*}
Moreover, while there may be infinitely many optimal pmfs $\mu^*$ for the right-hand side, there is a unique optimal $\eta^*=\mathcal{N}^T\mu^*$ that is related to the optimal $\rho^*$ from $p$-modulus as follows.
\begin{equation*}
\frac{\sigma(e)\rho^*(e)^p}{\mathcal{E}_{p,\sigma}(\rho^*)} =
\frac{\hat{\sigma}(e)\eta^*(e)^q}{\mathcal{E}_{q,\hat{\sigma}}(\eta^*)}.
\end{equation*}
\end{theorem}

When $\sigma\equiv 1$, $p=2$ and $\mathcal{N}$ has the form~\eqref{eq:natural-N}, this probabilistic interpretation has a particularly intuitive meaning.  Given any pmf $\mu$, we may consider two independent, identically distributed random objects $\rv{\gamma},\rv{\gamma}'\sim\mu$.  If we think of these two objects as subsets of $E$, then it can be seen that
\begin{equation*}
    \sum_{e\in E}\eta(e)^2 = \mathbb{E}_\mu(|\rv{\gamma}\cap\rv{\gamma}'|),
\end{equation*}
the expected size of the intersection between the two objects.  This is often called the \emph{expected overlap} of the objects.  In this setting, the $p$-modulus problem is equivalent to choosing a pmf $\mu$ so that this expected overlap is minimized.

There are also versions of Theorem~\ref{thm:prob-interp} for the cases $p\in\{1,\infty\}$.  These are necessarily weaker due to lack of uniqueness.

\begin{theorem}
\label{thm:prob-interp-special}
Let $\Gamma$ be a family of objects on $G$, then
\begin{equation*}
\Mod_{1,\sigma}(\Gamma)^{-1} = 
\min_{\substack{\mu\in\mathcal{P}(\Gamma)\\\eta=\mathcal{N}^T\mu}}
\mathcal{E}_{\infty,\sigma^{-1}}(\eta)\qquad\text{and}\qquad
\Mod_{\infty,\sigma}(\Gamma)^{-1} = 
\min_{\substack{\mu\in\mathcal{P}(\Gamma)\\\eta=\mathcal{N}^T\mu}}
\mathcal{E}_{1,\sigma^{-1}}(\eta).
\end{equation*}
\end{theorem}

Using the probabilistic interpretation, we can gain more insight into the earlier observation that modulus measures a balance between diversity and shortness.  To see this, consider an arbitrary family $\Gamma$ and note that the constant density $\rho=(\min\mathcal{N}\ones)^{-1}\ones$ (the reciprocal of the minimum row sum of $\mathcal{N}$) is admissible.  For $p\in(1,\infty)$, this implies that
\begin{equation*}
\Mod_{p,\sigma}(\Gamma) \le
\frac{\sigma(E)}{(\min\mathcal{N}\ones)^p},
\end{equation*}
where
\begin{equation*}
\sigma(E) = \sum\limits_{e\in E}\sigma(e).
\end{equation*}
When $\mathcal{N}$ has large row sums (corresponding to ``long'' or ``heavy'' objects), this forces the modulus to be small.

On the other hand, the uniform distribution $\mu=|\Gamma|^{-1}\ones$ is one possible pmf for Theorem~\ref{thm:prob-interp}, which provides a lower bound on modulus.  With this choice,
\begin{equation*}
\eta(e) = |\Gamma|^{-1}(\mathcal{N}^T\ones)(e)
\le |\Gamma|^{-1}\max\mathcal{N}^T\ones,
\end{equation*}
where $\max\mathcal{N}^T\ones$ is the maximum column sum of $\mathcal{N}$.  Thus,
\begin{equation*}
\Mod_{p,\sigma}(\Gamma)^{-1} \le
\frac{\hat{\sigma}(E)^{\frac{p}{q}}}{|\Gamma|^p}(\max\mathcal{N}^T\ones)^p.
\end{equation*}
When $\mathcal{N}$ has small column sums (corresponding to edges that are not used by many objects), this forces the modulus to be large.

%
%



\section{Modulus on temporal graphs}
\label{sec:mod-trp}

Having reviewed the definition of modulus and some of its interpretations, we now focus our attention on extending the $p$-modulus framework to objects on temporal graphs.

\subsection{A model for temporal networks}

In order to demonstrate the use of $p$-modulus on temporal networks, we must first choose a suitable network model.  For this paper, we have chosen the ``contact sequence'' graph described in~\cite[Sec.~3]{holme_temporal_networks_2012}.  Conceptually, this type of graph can be represented as a quadruple $G=(V,E,\sigma,T)$, where the added property $T$ encodes the temporal information.  For this paper, we will consider $T$ as a function $T:E\to 2^{\mathbb{R}}$; to each edge, $e\in E$, $T(e)$ assigns a discrete, positive set of times $t\in\mathbb{R}_{>0}$ during which $e$ is available for use.  To simplify the discussion, we shall assume that $T(e)$ is finite for each $e$.  The associated graph $\check{G}=(V,E,\sigma)$ with temporal information stripped away is often called the \emph{static graph} or \emph{aggregated graph}.


\subsection{Objects on temporal networks}

The next step in extending $p$-modulus to temporal networks is to select a family of objects, $\Gamma$.  Intuitively, each object should be somehow associated with the graph $G$ and should take into account the temporal information.  As in the non-temporal setting, it is evident that there is a great deal of flexibility in defining an object.  A fairly broad class of families can be obtained by considering subsets of edge-time pairs.  A concrete example of such a family is the family of \emph{time respecting paths} between $s$ and $t$ as described in~\cite[Sec.~4.2]{holme_temporal_networks_2012}.  These are sets of the form $\{(e_1,t_1),(e_2,t_2),\ldots,(e_k,t_k)\}$ where the edges $e_1e_2\cdots e_k$ form a path (traversed in order) from $s$ to $t$, and the times satisfy $t_1<t_2<\cdots<t_k$ and $t_i\in T(e_i)$ for all $i$.  In this example, one sees the connection between the underlying graph structure (the edges must trace a path from $s$ to $t$) and the temporal data (the edges must be traversed in a prescribed temporal order).

\subsection{Density and energy}

There are two relatively natural ways to define the concept of density on a temporal graph.  The first is to assign a value $\rho(e)$ to each edge $e\in E$.  The second is to assign a value $\rho(e,t)$ to each edge-time pair.  A practical way to accommodate either option, or even a mixture of both, is to allow $G$ to have parallel edges.  In other words, $G$ (and, therefore, $\check{G}$) is a multigraph with edge set $E$.  Each edge $e$ has a set $T(e)$ of one or more availability times and is assigned a value $\rho(e)$ by a density $\rho\in\mathbb{R}^E_{\ge 0}$.
In this way, the definition of energy~\eqref{eq:energy} remains unchanged in the temporal case.  All that remains in order to define a temporal version of modulus, then, is to find a suitable definition for the usage matrix $\mathcal{N}\in\mathbb{R}^{\Gamma\times E}$.

\subsection{Usage for temporal objects}
\label{sec:temporal-usage}

Here, we describe possible strategies for defining a temporal version of usage, assuming a particular structure on the family of temporal objects.  Our goal is not to develop a single generic approach for all families of temporal objects, but instead to provide the basis for a temporal $p$-modulus framework by way of example.

Specifically, we shall consider a family $\Gamma$ whose objects can be identified with sets of the form
\begin{equation*}
    \gamma = \{(e_1,t_1),(e_2,t_2),\ldots,(e_k,t_k)\}\quad\text{with}\quad
    t_i\in T(e_i)\text{ for }i=1,2,\ldots,k.
\end{equation*}
In words, the objects we consider will be sets of edge-time pairs with the requirement that the time occurring in the pair is one of the available times of the corresponding edge.  To simplify the discussion, we shall also assume that no edge appears more than once in a given object (i.e., that the edges $e_i$ are unique in $\gamma$).  This class of family is a natural generalization of the time respecting paths.

Associated with each object $\gamma\in\Gamma$ is a natural projection $\check{\gamma}\in 2^E$ defined as
\begin{equation*}
    \check{\gamma} = \{e_1,e_2,\ldots,e_k\}.
\end{equation*}
Since we assumed that $\gamma$ cannot have repeated edges, it follows that $|\check{\gamma}|=|\gamma|$.  In this way, we can think of $\check{\gamma}$ as a projected object on the aggregated graph $\check{G}$.  The collection of all projections of temporal objects in $\Gamma$ forms a family of objects on $\check{G}$, which we denote $\check{\Gamma}$.  That is,
\begin{equation*}
\check{\Gamma} = \{\check{\gamma}:\gamma\in\Gamma\}.
\end{equation*}

Since each $\check{\gamma}\in\check{\Gamma}$ is a subset of $E$, a natural choice of usage matrix for this aggregated family is $\check{\mathcal{N}}$ of the form~\eqref{eq:natural-N}.  We seek to define a usage matrix $\mathcal{N}$ on $\Gamma$ by incorporating temporal information into $\check{\mathcal{N}}$.  We shall do this by means of a strictly monotonic \emph{temporal penalty function} (or \emph{time-penalty function}) $\varphi:\mathbb{R}\to\mathbb{R}_{\ge 0}$.  In general, we will think of $\varphi$ as monotonically increasing, though it is possible to envision applications wherein a monotonically decreasing $\varphi$ would be of interest.

\subsubsection{Per-edge time penalization}

In order to find a suitable usage matrix, one might consider modifying the concept of $\rho$-length by incorporating temporal information independently on each edge.  For example, given a density $\rho$ and an object $\gamma\in\Gamma$, one could define
\begin{equation*}
\ell_\rho(\gamma) = 
\sum_{i=1}^k f(\rho(e_i),\varphi(t_i)).
\end{equation*}
In order to apply the established $p$-modulus apparatus to the family of temporal objects, we need to find a usage matrix $\mathcal{N}\in\mathbb{R}^{\Gamma\times E}_{\ge 0}$ with the property that
\begin{equation*}
\ell_\rho(\gamma)\ge 1\quad\iff\quad
(\mathcal{N}\rho)(\gamma)\ge 1.
\end{equation*}
There is no reason to expect such a matrix to exist in general, but for certain choices of $f$ it does.

For example, if $f$ is the multiplication operator, then
\begin{equation*}
\ell_\rho(\gamma) = \sum_{i=1}^k\rho(e_i)\varphi(t_i)
= \sum_{e\in E}\mathcal{N}(\gamma,e)\rho(e),
\end{equation*}
where
\begin{equation}
    \label{eq:N-per-edge-mult}
    \mathcal{N}(\gamma,e) = 
    \begin{cases}
        \varphi(t)&\text{if }(e,t)\in\gamma,\\
        0 &\text{if }e\notin\check{\gamma}.
    \end{cases}
\end{equation}

If $f$ instead is the addition operator then
\begin{equation*}
\ell_\rho(\gamma) =
\ell_\rho(\check{\gamma}) + \sum_{i=1}^k\varphi(t_i).
\end{equation*}
Requiring that $\ell_\rho(\gamma)\ge 1$ is equivalent to requiring that
\begin{equation*}
    \ell_\rho(\check{\gamma})\ge
    1 - \sum_{i=1}^k\varphi(t_i).
\end{equation*}
Since $\rho$ and $\varphi$ are non-negative, we may assume without loss of generality (by possibly removing objects from $\Gamma$) that the right-hand side of the inequality is strictly positive for all $\check{\gamma}\in\check{\Gamma}$.  If we define
\begin{equation*}
\varphi(\gamma) = \sum_{i=1}^k\varphi(t_i),
\end{equation*}
then the minimum $\rho$-length constraint can be written as
\begin{equation*}
1 \le \frac{\ell_\rho(\check{\gamma})}{1-\varphi(\gamma)} = \sum_{e\in E}\frac{\check{\mathcal{N}}(\check{\gamma},e)}{1-\varphi(\gamma)}\rho(e),
\end{equation*}
which provides the temporal usage matrix
\begin{equation*}
\mathcal{N}(\gamma,e) = 
\frac{\check{\mathcal{N}}(\check{\gamma},e)}{1-\varphi(\gamma)}.
\end{equation*}

\subsubsection{Per-object time penalization}

Alternatively, one might assign a time penalty to the object as a whole and combine this with $\rho$-length on the aggregated graph.  As an example, we could define a length of the form
\begin{equation*}
\ell_\rho(\gamma) = f(\ell_\rho(\check{\gamma}),\varphi(\gamma)),\qquad\text{where}\qquad
\varphi(\gamma) = \max_i\varphi(t_i).
\end{equation*}

Natural assumptions on $f$ would be non-negativity and monotonicity in each argument---larger $\rho$-length in the aggregated graph and larger time penalties both induce larger temporal $\rho$-length.  Moreover, in order to avoid the trivial minimizer $\rho^*\equiv 0$ and to ensure that the set of admissible densities is non-empty, we will assume the following.
\begin{enumerate}
    \item For all $\gamma\in\Gamma$, $f(0,\varphi(\gamma))<1$
    \item For all $\gamma\in\Gamma$, 
    $f(\ell,\varphi(\gamma))\ge 1$ for sufficiently large $\ell$.
\end{enumerate}
If we further require that $f$ be right continuous in its first argument, then to every $\gamma\in\Gamma$ we may assign a critical length $\ell^*(\gamma)>0$ so that
\begin{equation*}
    \ell_\rho(\gamma) \ge 1\quad\iff\quad
    \ell_\rho(\check{\gamma}) \ge \ell^*(\gamma).
\end{equation*}

In this case, the temporal usage matrix, $\mathcal{N}$, that we seek can be found from the minimum length inequalities, since $\ell_\rho(\gamma)\ge 1$ if and only if
\begin{equation*}
1 \le \frac{\ell_\rho(\check{\gamma})}{\ell^*(\gamma)}
= \sum_{e\in E}\frac{\check{\mathcal{N}}(\check{\gamma},e)}{\ell^*(\gamma)}\rho(e),
\end{equation*}
so we should choose
\begin{equation*}
\mathcal{N}(\gamma,e) = \frac{\check{\mathcal{N}}(\check{\gamma},e)}{\ell^*(\gamma)}.
\end{equation*}

In the particular case that $f$ is the multiplication operator, we obtain
\begin{equation}
\label{eq:N-per-obj-mult}
\mathcal{N}(\gamma,e) = \varphi(\gamma)\check{\mathcal{N}}(\check{\gamma},e).
\end{equation}
If $f$ is the addition operator, then the usage matrix is
\begin{equation*}
\mathcal{N}(\gamma,e) = 
\frac{\check{\mathcal{N}}(\check{\gamma},e)}
{1-\varphi(\gamma)},
\end{equation*}
under the additional assumption that $\varphi(\gamma)<1$ for all $\gamma\in\Gamma$.

\subsubsection{Some general assumptions}

All of the above possibilities can be realized by choosing $\mathcal{N}$ of the form
\begin{equation}\label{eq:generic-N}
\mathcal{N}(\gamma,e) =
\psi(\gamma,e)\check{\mathcal{N}}(\check{\gamma},e),
\end{equation}
where \begin{equation}\label{eq:temp-psi}
    \psi(\gamma,e) :=
    \begin{cases}
    \varphi(t) & \text{if }(e,t)\in\gamma,\\
    1 &\text{if }e\notin\check{\gamma}
    \end{cases}
\end{equation} performs some type of ``temporal weighting.'' For the remainder of this paper, we shall consider this as a general form of temporal usage matrix.

By making some basic assumptions on this function $\psi$, it is possible to translate a large portion of the modulus theory developed for non-temporal graphs to the temporal setting.  For the remainder of this paper, we shall make the following assumptions about $\psi$.

\begin{asm}
The function $\psi$ is \emph{universal} in the sense that it is defined for any set $\gamma$ of the form
\begin{equation*}
\gamma = \{(e_1,t_1),(e_2,t_2),\ldots,(e_k,t_k)\}\quad\text{with}\quad
e_i\in E\text{ and }t_i>0\text{ for }i=1,2,\ldots,k,
\end{equation*}
and any $e\in E$.  This allows us to consider $\psi$ separately from the particular choice of temporal object family $\Gamma$.
\end{asm}

\begin{asm}
The function $\psi$ is monotonic with respect to inclusions.  If $\gamma'\subseteq\gamma$, then $\psi(\gamma',e) \le \psi(\gamma,e)$.
\end{asm}

\begin{asm}
The function $\psi$ is monotonic with respect to time.  If $\gamma=\{(e_1,t_1),\cdots,(e_k,t_k)\}$ and $\gamma'=\{(e_1,t_1'),\cdots,(e_k,t_k')\}$ with $t_i'\le t_i$ for all $i$, then $\psi(\gamma',e) \le \psi(\gamma,e)$.
\end{asm}

\begin{asm}
If $\gamma^\lambda=\{(e_1,\lambda t_1),\cdots,(e_k,\lambda t_k)\}$ (parametrized by $\lambda>0$) then the function $\lambda\mapsto\psi(\gamma^\lambda,e)$ is continuous for each $e$ and
\begin{equation*}
\lim_{\lambda\to 0^+}\psi(\gamma^\lambda,e) = 1.
\end{equation*}
\end{asm}

For the usage examples described previously, these assumptions can be satisfied by suitable assumptions on the function $\varphi$.  In the following, we consider a generic set $\gamma=\{(e_1,t_1),\ldots,(e_k,t_k)\}$.

\paragraph{Multiplicative per-edge penalization.} In this case,
\begin{equation}\label{eq:mpe-penalty}
\psi(\gamma,e) =
\begin{cases}
\varphi(t) &\text{if }(e,t)\in\gamma,\\
1 &\text{if }e\notin\check{\gamma}.
\end{cases}
\end{equation}
The assumptions are met if $\varphi$ is a continuous, increasing function satisfying $\varphi(0)=1$.

\paragraph{Additive per-edge penalization.} In this case,
\begin{equation*}
\psi(\gamma,e) = 
\begin{cases}
\left(1 - \sum\limits_{i=1}^k\varphi(t_i)\right)^{-1}
&\text{if }e\in\check{\gamma},\\
1 &\text{if }e\notin\check{\gamma}.
\end{cases}
\end{equation*}
(Special care must be taken if the sum equals or exceeds 1.)
The assumptions are met if $\varphi$ is a continuous, increasing function satisfying $\varphi(0)=0$.

\paragraph{Multiplicative per-object penalization.} In this case,
\begin{equation}\label{eq:mpo-penalty}
\psi(\gamma,e) =
\begin{cases}
\max\varphi(t_i)&\text{if }e\in\check{\gamma},\\
1 &\text{if }e\notin\check{\gamma}.
\end{cases}
\end{equation}
The assumptions are met if $\varphi$ is a continuous, increasing function satisfying $\varphi(0)=1$.

\paragraph{Additive per-object penalization.} In this case,
\begin{equation*}
\psi(\gamma,e) = 
\begin{cases}
\left(1 - \max\varphi(t_i)\right)^{-1}
&\text{if }e\in\check{\gamma},\\
1 &\text{if }e\notin\check{\gamma}.
\end{cases}
\end{equation*}
(Special care must be taken if the maximum equals or exceeds 1.)
The assumptions are met if $\varphi$ is a continuous, increasing function satisfying $\varphi(0)=0$.

\subsection{Modulus of a family of temporal objects}

Using the procedure outlined above, we have now developed the concepts of density and energy on a temporal graph $G$ and the concept of usage on a temporal object $\gamma\in\Gamma$. With these definition in hand, the $p$-modulus of $\Gamma$ can be defined exactly as in~\eqref{eq:DefMod}.  In this way, we have incorporated all needed temporal information into the usage matrix $\mathcal{N}$ and, when it is convenient, we may think of $\Gamma$ as a family of objects on the aggregated graph $\check{G}$.  This allows properties of modulus to translate more or less directly from the non-temporal to the temporal case.

\subsection{A mass transport interpretation of modulus}

To help provide some intuition about what temporal modulus measures, we now describe a mass transport interpretation based on the probabilistic interpretation of modulus described in Section~\ref{sec:prob-interp}.  For this discussion, we focus on the family $\Gamma=\Gamma^{\trp}(s,t)$ of time respecting paths passing from a vertex $s$ to another vertex $t$ in a temporal graph $G$.  Note that, in this case, the projected family $\check{\Gamma}$ is a subset of $\Gamma(s,t)$, the family of paths connecting $s$ to $t$ in $\check{G}$.  This inclusion is strict if there are paths $\check{\gamma}\in\check{\Gamma}$ that cannot be traced while respecting the temporal constraint.  Moreover, $|\check{\Gamma}|\le|\Gamma|$, and this inequality is strict if there are multiple ways to trace some path while respecting time.  For a temporal usage matrix, we adopt an $\mathcal{N}$ of the form~\eqref{eq:generic-N}, with $\check{\mathcal{N}}$ as in~\eqref{eq:natural-N}.

Every pmf $\mu \in \mathcal{P}(\Gamma) $ can be thought of as a \emph{plan} for transporting one unit of mass along the time
respecting paths from $s$ to $t$. For any path $\gamma \in \Gamma$, $\mu(\gamma)$ provides the fraction of mass that should
be sent along $\gamma$. In this way, we can view temporal modulus within the class of optimal mass transport problems, with $p$, $\sigma$ and $\psi$ defining the fitness function as follows.

First, every plan incurs a \emph{per edge usage cost}
\begin{equation*}
\eta(e) = \sum_{\gamma\in\Gamma}\mathcal{N}(\gamma,e)\mu(\gamma)
= \sum_{\gamma\in\Gamma}\psi(\gamma,e)\check{\mathcal{N}}(\gamma,e)\mu(\gamma)
= \sum_{\gamma\in\Gamma_e}\psi(\gamma,e)\mu(\gamma),
\end{equation*}
where $\Gamma_e = \{\gamma\in\Gamma : e\in\check{\gamma}\}$ is the subfamily of paths that traverse edge $e$.  The quantity $\psi(\gamma,e)$ can be thought of as an amount of some resource per unit mass consumed when transporting this unit mass along edge $e$ of $\gamma$.  Examples of $\psi$ can be found in Section~\ref{sec:temporal-usage}.  Since the temporal penalty function $\varphi$ is assumed to be increasing, paths that cross edges a later times tend to consume more resource per unit mass than paths traversing edges at earlier times. The quantity $\eta(e)$, then, is the accumulated resource expenditure on edge $e$ required to enact the plan $\mu$.

The values $p$ and $\sigma$ then define the global fitness function $\mathcal{E}_{q,\hat{\sigma}}(\eta)$ that we seek to minimize by choosing the plan $\mu$. For example, if we simply wish to minimize the total resource expenditure (i.e., the sum of $\eta(e)$ across all edges),
this corresponds to setting $p=\infty$ and $\sigma \equiv 1$. On the other hand, if we wish to minimize the
maximum resource expenditure across all edges, then we would set $p=1$ and $\sigma \equiv 1$. More
generally, $\sigma(e)$ allows us to adjust the importance of edge $e$ within the fitness function and $p$
provides us a way of interpolating between the extremes of minimizing the sum of $\sigma^{-1}(e)\eta(e)$ and
minimizing the maximum of $\sigma^{-1}(e)\eta(e)$. In particular, 2-modulus corresponds to minimizing the
sum of $\sigma^{-1}(e)\eta(e)^2$
across all edges and bears a striking resemblance to effective resistance if we think of $\eta(e)$ as a generalized ``current'' flowing across edge $e$.

\subsection{Properties of temporal modulus}
\label{sec:properties}

In this section, we explore some of the properties of modulus as just defined.

\subsection{Properties inherited from static modulus}

As remarked earlier, a consequence of the preceding definition of temporal modulus is that all temporal information is encoded into $\Gamma$ and $\mathcal{N}$, allowing us to think of $\Gamma$ as a general family of objects on $\check{G}$.  Every row of $\mathcal{N}$ corresponds to some path $\check{\gamma}\in\Gamma(s,t)$, but with nonstandard edge usage induced by the temporal penalty.  If there are multiple ways to trace $\check{\gamma}$ while respecting time, there will be multiple corresponding rows in $\mathcal{N}$.  Similarly, not every path from $s$ to $t$ in $\check{G}$ is necessarily represented in $\mathcal{N}$.

A consequence of this is that the general theory of $p$-modulus can be directly applied to the temporal modulus problem.  As important examples, the following theorems are immediate (see~\cite{ProbInterp2016,Blocking2018}).

\begin{theorem}\label{thm:trp-monotone}
For $\Gamma=\Gamma^{\trp}(s,t)$, the function $p\mapsto\Mod_{p,\sigma}(\Gamma)$ is continuous on $[1,\infty)$.  Moreover, with
\begin{equation*}
\sigma(E) := \sum_{e\in E}\sigma(e),
\end{equation*}
the functions
\begin{equation*}
p\mapsto \phi(1)^p\Mod_{p,\sigma}(\Gamma)\qquad\text{and}\qquad
p\mapsto \left(
\sigma(E)^{-1}\Mod_{p,\sigma}(\Gamma)
\right)^{\frac{1}{p}}
\end{equation*}
are nonincreasing and nondecreasing respectively.
\end{theorem}

\begin{theorem}\label{thm:temp-prob-interp}
Theorems~\ref{thm:prob-interp} and~\ref{thm:prob-interp-special} apply when $\Gamma=\Gamma^{\trp}(s,t)$, provided $\mathcal{N}$ is understood in the sense of~\eqref{eq:generic-N} and~\eqref{eq:temp-psi}.
\end{theorem}

\begin{theorem}\label{thm:temp-sigma}
Let $\Gamma=\Gamma^{\trp}(s,t)$ and let $p\in(1,\infty)$.  For any $\sigma\in\mathbb{R}^E_{>0}$, let $\rho_\sigma^*$ be the unique optimal density for $\Mod_{p,\sigma}(\Gamma)$ and let $\eta_\sigma^*$ be the corresponding unique optimal edge usages from the associated mass transport problem.  Then the following are true.
\begin{enumerate}
\item The function $\sigma\mapsto\rho_\sigma^*$ is continuous.
\item The function $F(\sigma):=\Mod_{p,\sigma}(\Gamma)$ is concave and differentiable.  Its partial derivatives satisfy
\begin{equation*}
    \frac{\partial F}{\partial\sigma(e)} =
    \rho_\sigma^*(e)^p\qquad\text{for all }e\in E.
\end{equation*}
\item In each coordinate direction $e\in E$, the functions $\sigma\mapsto\Mod_{p,\sigma}(\Gamma)$ and 
$\sigma\mapsto\eta_\sigma^*(e)$ are nondecreasing while $\sigma\mapsto\rho_\sigma^*(e)$ nonincreasing.  In other words, increasing $\sigma$ on a single edge $e$ cannot decrease the modulus nor the value $\eta_\sigma^*(e)$ and cannot increase $\rho_\sigma^*(e)$.
\end{enumerate}
\end{theorem}

\subsection{Temporal scaling}

One might expect that, for a temporal version of modulus to be useful, it should be in some way comparable to the static version of modulus.  As an example, fix a family of connecting time respecting paths, $\Gamma=\Gamma^{\trp}(s,t)$, on a temporal graph along with a temporal penalty function $\phi$. We consider a family of temporal modulus problems $\Mod_{p,\sigma}^{\lambda}(\Gamma)$ parametrized by $\lambda>0$, which are identical except in the choice of temporal penalty function, $\phi^\lambda$, defined by $\phi^\lambda(t)=\phi(\lambda t)$. This, in turn, defines a parametrized family of $\rho$-lengths,
\begin{equation*}
\ell_\rho^\lambda(\gamma) = 
\sum_{i=1}^k\rho(e_i)\phi(\lambda t_i).
\end{equation*}
As $\lambda$ approaches zero, one might expect to obtain something that looks like a static modulus problem.  And, indeed, this is what happens.

\begin{lem}
\label{lem:param-bounds}
For the $\lambda$-parametrized problem described in this section, define
\begin{equation*}
    M = \max_{e\in E}\max T(e).
\end{equation*}
Then, for all $p\in[1,\infty)$,
\begin{equation*}
\frac{\Mod_{p,\sigma}(\check{\Gamma})}{\phi(\lambda M)^{p}} \le
\Mod_{p,\sigma}^{\lambda}(\Gamma) \le
\Mod_{p,\sigma}(\check{\Gamma}).
\end{equation*}
Also,
\begin{equation*}
\frac{\Mod_{\infty,\sigma}(\check{\Gamma})}{\phi(\lambda M)} \le
\Mod_{\infty,\sigma}^{\lambda}(\Gamma) \le
\Mod_{\infty,\sigma}(\check{\Gamma}).
\end{equation*}
\end{lem}

\begin{proof}
First, let $\rho^*$ be an optimal density for $\Mod_{p,\sigma}(\check{\Gamma})$ and let $\gamma\in\Gamma^{\trp}(s,t)$.  Then, since $\rho^*$ is admissible for the family $\check{\Gamma}$,
\begin{equation*}
\ell_{\rho^*}^\lambda(\gamma)
= \sum_{i=1}^k\rho^*(e_i)\phi(\lambda t_i)
\ge \sum_{i=1}^k\rho^*(e_i)
= \ell_\rho(\check{\gamma})\ge 1
\end{equation*}
so $\rho^*$ is also admissible for the $\lambda$-parametrized problem.  Therefore,
\begin{equation*}
\Mod_{p,\sigma}^{\lambda}(\Gamma) \le
\mathcal{E}_{p,\sigma}(\rho^*)
= \Mod_{p,\sigma}(\check{\Gamma}).
\end{equation*}

Now, suppose that $\rho^*$ is optimal for $\Mod_{p,\sigma}^{\lambda}(\Gamma)$ and that $\check{\gamma}\in\check{\Gamma}$.  By definition, there is a time respecting path $\gamma\in\Gamma$ whose projection is $\check{\gamma}$.  Since $\rho^*$ is admissible for the $\lambda$-parametrized problem,
\begin{equation*}
\ell_{\rho^*}(\check{\gamma}) =
\sum_{i=1}^k\rho^*(e_i) =
\sum_{i=1}^k\frac{\phi(\lambda t_i)}{\phi(\lambda t_i)}\rho^*(e_i) \ge
\frac{1}{\phi(\lambda M)}
\sum_{i=1}^k\phi(\lambda t_i)\rho^*(e_i)
=\frac{\ell^{\lambda}_{\rho^*}(\gamma)}{\phi(\lambda M)}
\ge \frac{1}{\phi(\lambda M)}.
\end{equation*}
Thus, $\rho = \phi(\lambda M)\rho^*$ is admissible for the family $\check{\Gamma}$ and so
\begin{equation*}
\Mod_{p,\sigma}(\check{\Gamma})
\le \mathcal{E}_{p,\sigma}(\rho)
= \mathcal{E}_{p,\sigma}(\phi(\lambda M)\rho^*).
\end{equation*}
The theorem follows from the scaling properties of $\mathcal{E}_{p,\sigma}$ and the optimality of $\rho^*$.
\end{proof}

\begin{theorem}\label{thm:temp-mod-limit}
In the setting of Lemma~\ref{lem:param-bounds},
Let $p\in[1,\infty]$.  The function $\lambda\mapsto\Mod_{p,\sigma}^{\lambda}(\Gamma)$ is monotonically nonincreasing and
\begin{equation}\label{eq:param-mod-conv}
\lim_{\lambda\to 0^+}\Mod_{p,\sigma}^{\lambda}(\Gamma)
=\Mod_{p,\sigma}(\check{\Gamma}).
\end{equation}
Moreover, suppose $p\in(1,\infty)$ and let $\rho^\lambda$ be the unique optimal density for $\Mod_{p,\sigma}^{\lambda}(\Gamma)$ and $\check{\rho}$ the unique optimal density for $\Mod_{p,\sigma}(\check{\Gamma})$.  Then
\begin{equation}\label{eq:param-rho-conv}
\lim_{\lambda\to 0^+}\rho^\lambda = \check{\rho}.
\end{equation}
\end{theorem}

\begin{proof}
The limit~\eqref{eq:param-mod-conv} follows directly from the lemma.  To obtain the monotonicity, let $\lambda_2\ge\lambda_1>0$ and let $\rho_1$ be an optimal density for $\Mod_{p,\sigma}^{\lambda_1}(\Gamma)$.  Suppose $\gamma\in\Gamma^{\trp}(s,t)$.  Then, since $\rho_1$ is admissible for the case $\lambda=\lambda_1$,
\begin{equation*}
1 \le \ell_{\rho_1}^{\lambda_1}(\gamma) =
\sum_{i=1}^k\rho_1(e_i)\phi(\lambda_1 t_i) \le
\sum_{i=1}^k\rho_1(e_i)\phi(\lambda_2 t_i) =
\ell_{\rho_1}^{\lambda_2}(\gamma),
\end{equation*}
so $\rho_1$ is also admissible in the case $\lambda=\lambda_2$. Therefore,
\begin{equation*}
\Mod_{p,\sigma}^{\lambda_2}(\Gamma) \le
\mathcal{E}_{p,\sigma}(\rho_1)
= \Mod_{p,\sigma}^{\lambda_1}(\Gamma).
\end{equation*}
The pointwise convergence in~\eqref{eq:param-rho-conv} is a consequence of Clarkson's inequalities and the energy convergence (see~\cite[Lemma~6.2]{albin2015modulus}).
\end{proof}

\section{Examples}
\label{sec:examples}
The following examples demonstrate the ability for $p$-modulus to incorporate temporal information.  In the examples, we will consider the two types of temporal usage matrix, those corresponding to the the per-edge multiplicative penalty~\eqref{eq:mpe-penalty} and the multiplicative per-object penalty~\eqref{eq:mpo-penalty}.  Further examples can be found in~\cite{mikheev2020modulus}.

\begin{exmp}
\label{ex:multiedge}

\begin{figure}
\centering
\includegraphics[scale=0.3]{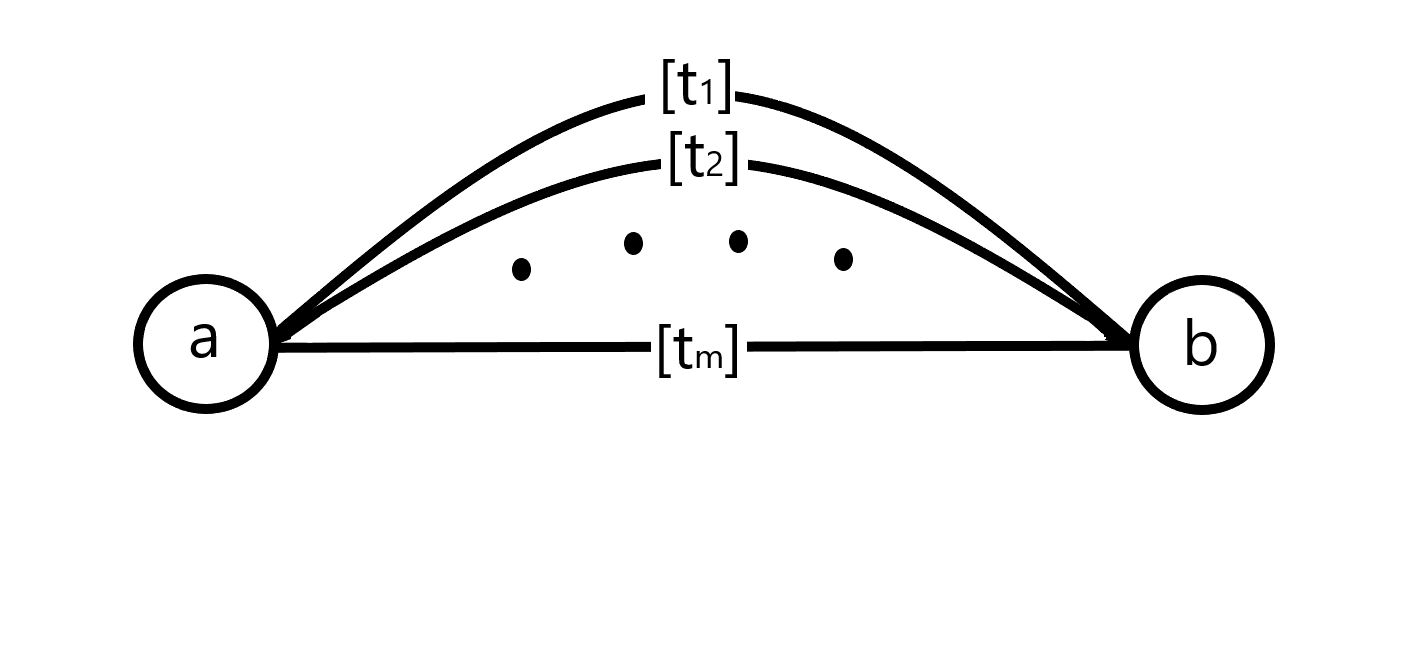}
\caption{Temporal graph for Example~\ref{ex:multiedge}.}
\label{fig:multiedge}
\end{figure}

Consider an undirected graph $G$ with two vertices, $a$ and $b$, and a collection of parallel edges $\{e_1,e_2,\ldots,e_m\}$, each with a single available time $T(e_i)=\{t_i\}$ (see Figure \ref{fig:multiedge}).  In this case,~\eqref{eq:mpe-penalty} and~\eqref{eq:mpo-penalty} give the same usage matrix.  The length constraints on $\rho$ are
\begin{equation*}
    \varphi(t_i)\rho(e_i) \ge 1\quad i=1,2,\ldots,m.
\end{equation*}
Since we seek to minimize $\mathcal{E}_{p,\sigma}(\rho)$, it is straightforward to check that $\rho^*(e_i)=\varphi(t_i)^{-1}$ is an optimal density, so
\begin{equation*}
\Mod_{p,\sigma}(\Gamma^{\trp}(a,b))
=
\begin{cases}
\sum\limits_{i=1}^m\frac{\sigma(e_i)}{\varphi(t_i)^p}
&\text{if }p<\infty,\\
\max\limits_i\frac{\sigma(e_i)}{\varphi(t_i)}
&\text{if }p=\infty.
\end{cases}
\end{equation*}
Each object in this example is a single edge-time pair, $\gamma_i=\{(e_i,t_i)\}$ for $i=1,2,\ldots,m$.  From Theorem~\ref{thm:prob-interp} one can see that the optimal mass transport plan when $p\in(1,\infty)$ assigns $\gamma_i$ a mass proportional to $\sigma(e_i)\varphi(t_i)^{-\frac{p}{q}}$.  For $p\approx 1$, this mass is primarily determined by the relative sizes of the $\sigma(e_i)$---larger capacity edges will be assigned more mass regardless the availability time.  For $p\gg 1$, it is primarily determined by the relative sizes of the $\varphi(t_i)$---edges that are available at early times will be assigned more mass regardless their capacity.
\end{exmp}

\begin{exmp}
\label{ex:consec}
Consider the undirected path graph $P_n$ with $n$ vertices and $\ell=n-1$ edges $\{e_1,e_2,\ldots,e_\ell\}$ which, when traversed in order, connect the vertex $a$ to $b$.  Assume that each edge has a single available time $T(e_i)=\{t_i\}$ and that $t_1<t_2<\cdots<t_\ell$ so that $\Gamma^{\trp}(a,b)$ contains a single time respecting path $\gamma=\{(e_1,t_1),(e_2,t_2),\ldots,(e_\ell,t_\ell)\}$.  Since there is a single path, the optimal mass transport plan is to assign all mass to this path.  Thus, the optimal expected edge usage is $\eta^*(e)=\mathcal{N}(\gamma,e)=\psi(\gamma,e)$.  For $p\in(1,\infty)$, Theorem~\ref{thm:prob-interp} shows that
\begin{equation*}
\Mod_{p,\sigma}(\Gamma^{\trp}(a,b))
= \left(
\sum_{i=1}^\ell\hat{\sigma}(e_i)\psi(\gamma,e_i)^q
\right)^{-\frac{p}{q}}.
\end{equation*}
If we choose per-edge temporal penalization, then
\begin{equation*}
\Mod_{p,\sigma}(\Gamma^{\trp}(a,b))
= \left(
\sum_{i=1}^\ell\hat{\sigma}(e_i)\varphi(t_i)^q
\right)^{-\frac{p}{q}}.
\end{equation*}
On the other hand, if we choose per-object temporal penalization, then
\begin{equation*}
\Mod_{p,\sigma}(\Gamma^{\trp}(a,b))
= \varphi(t_\ell)^{-p}\left(
\sum_{i=1}^\ell\hat{\sigma}(e_i)
\right)^{-\frac{p}{q}}.
\end{equation*}
Modulus decreases as the available times of the edges in the path increase.

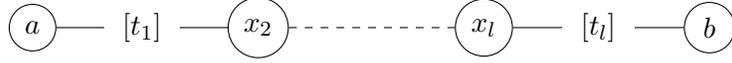
\begin{figure}
\centering
\begin{tikzpicture}[scale=3]
\foreach \name/\x/\y in {a/0/0,x_2/1/0,x_{l}/2/0,b/3/0} {
	\node[draw,circle] (\name) at (\x,\y) {$\name$};
}
 \foreach \u/\v/\t in {{a/x_2/{[t_1]}},{x_{l}/b/[t_{l}]}} {
 	\draw (\u) -- node[circle,fill=white] {$\t$} (\v);
}
 \foreach \u/\v in {x_2/x_{l}}{
 \draw[dashed] (\u) -- (\v);
 }
\end{tikzpicture}
\caption{Temporal graph for Example~\ref{ex:consec}.}
\label{fig:consec}
\end{figure}
\end{exmp}


\begin{exmp}
\label{ex:paral}
As a temporal version of the example in Section~\ref{sec:mod-example} (see Figure \ref{fig:paral}), consider $k$ identical copies of the path graph described in the previous example, all glued together at common end vertices $a$ and $b$.  From the symmetry of the configuration, it is evident that each of the $k$ paths will be assigned a mass of $1/k$ by the optimal transport plan, so for $e\in\gamma$,  $\eta^*(e)=\psi(\gamma,e)/k$.  Applying Theorem~\ref{thm:prob-interp} once again shows that
\begin{equation*}
\Mod_{p,\sigma}(\Gamma^{\trp}(a,b))
= \left(
k\sum_{i=1}^r \hat{\sigma}(e_i)\frac{\psi(\gamma,e_i)^q}{k^q}
\right)^{-\frac{p}{q}}.
\end{equation*}
If $\sigma\equiv 1$, then per-edge penalization and per-object penalization respectively yield
\begin{equation*}
\Mod_{p,1}(\Gamma^{\trp}(a,b))
=
k
\left(
\sum_{i=1}^r\varphi(t_i)^q
\right)^{-\frac{p}{q}}
\quad\text{and}\quad
\Mod_{p,1}(\Gamma^{\trp}(a,b))
= \frac{k}{r^{p-1}}\varphi(t_r)^{-p},
\end{equation*}
showing the triple interaction among the number of paths, the (graph) lengths of the paths, and the time required to traverse the paths on the temporal graph.

Since the modulus is continuous in $p$ (Theorem~\ref{thm:trp-monotone}) and since $\varphi$ is non-decreasing, $1$-modulus for both per-edge penalization and per-object one is
$$
{\rm Mod}_{1,1}(\Gamma^{\rm trp}(a,b))
=
\lim_{p\rightarrow 1 } {\rm Mod}_{p,1}(\Gamma^{\rm trp}(a,b))= k\varphi(t_r)^{-1},
$$
Similarly, $\infty$-modulus for per-edge penalization is 
$$
{\rm Mod}_{\infty,1}(\Gamma^{\rm trp}(a,b))
=
\lim_{p\rightarrow \infty } {\rm Mod}_{p,1}(\Gamma^{\rm trp}(a,b))^{\frac{1}{p}}= \left(\sum_{i=1}^r\varphi(t_i)\right)^{-1}.
$$
and for object-edge penalization is
$$
{\rm Mod}_{\infty,1}(\Gamma^{\rm trp}(a,b))
=
\lim_{p\rightarrow \infty } {\rm Mod}_{p,1}(\Gamma^{\rm trp}(a,b))^{\frac{1}{p}}= \frac{1}{r}\varphi(t_r)^{-1}.
$$

The optimal mass transport plan for the dual problem is the vector \ $\displaystyle \mu^*=\frac{1}{k} \mathbf{1}$; the unit mass is split evenly into $k$ parts, with each portion transported across one of the parallel paths.

\begin{figure}
\centering
\begin{tikzpicture}[scale=3]
\foreach \name/\x/\y in {a/0/0,x_{11}/0.85/0.72,x_{1r-1}/2.85/0.72,x_{2r-1}/2.85/0,x_{21}/0.85/0,b/4.2/0, x_{k1}/0.85/-0.72, x_{kr-1}/2.85/-0.72}{
	\node[draw,circle] (\name) at (\x,\y) {$\name$};
}
\foreach \u/\v/\t in {a/x_{11}/[t_1],{x_{1r-1}/b/[t_r]},a/x_{21}/[t_1],x_{2r-1}/b/[t_r],a/x_{k1}/[t_1],x_{kr-1}/b/[t_r]} {
	\draw (\u) -- node[circle,fill=white] {$\t$} (\v);
}
\foreach \u/\v in {x_{11}/x_{1r-1},x_{21}/x_{2r-1},x_{k1}/x_{kr-1}}{
\draw[dashed] (\u) -- (\v);
}
\end{tikzpicture}
\caption{Temporal graph for Example~\ref{ex:paral}.}
\label{fig:paral}
\end{figure}
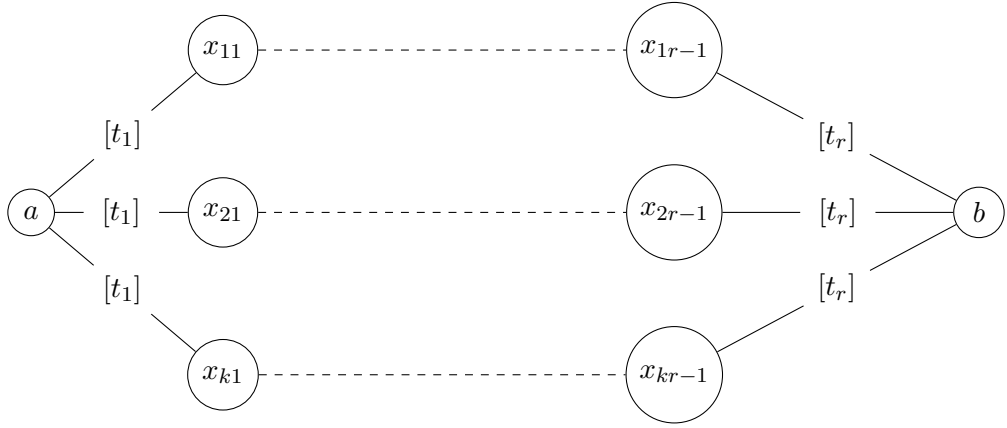
\end{exmp}


\begin{exmp}
\label{ex:pawgraph}
\begin{figure}
\centering
\begin{tikzpicture}[scale=4]
\foreach \name/\x/\y in {a/0/0,b/0.5/0.71,t/1/0,s/-1/0} {
	\node[draw,circle] (\name) at (\x,\y) {$\name$};
}
\foreach \u/\v/\t in {s/a/[1],a/b/[2],{b/t/[T']},a/t/[T]} {
	\draw (\u) -- node[circle,fill=white] {$\t$} (\v);
}
\end{tikzpicture}
\caption{Temporal graph for Example~\ref{ex:pawgraph}.}
\label{fig:pawgraph}
\end{figure}
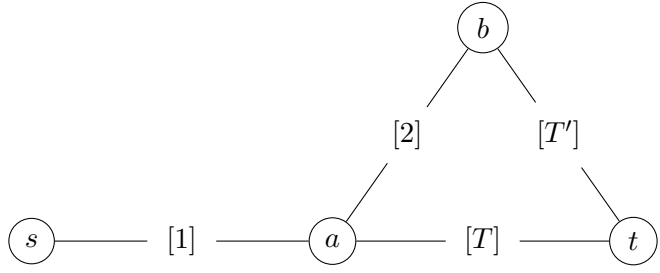

Next, consider the $2$-modulus of the set of time-respecting paths $\Gamma$ from $s$ to $t$ in the graph on Figure \ref{fig:pawgraph}, where all weights on the edges are assumed to be 1, and $\varphi(t)$ is a non-decreasing positive per-object penalization.

For convenience, we define $\varphi(T)=:\alpha$, \ $\varphi(T')=:\beta$. Assuming that $T'>2$, the usage matrix $\mathcal{N}$ can be written as
$$
\mathcal{N} =  \begin{bmatrix}
    \alpha & \alpha & 0 & 0 \\
    \beta & 0 & \beta & \beta 
  \end{bmatrix}.
$$
(The first row of $\mathcal{N}$ corresponds to the path from $s$ to $a$ to $t$ and the second row corresponds to the other time-respecting path.)

Let $\mu= \begin{bmatrix}
    \delta \\
    1-\delta
    \end{bmatrix}$
, be an arbitrary plan parameterized by $0\leq \delta \leq 1$.
Then
$$
{\rm Mod}_{2,1}(\Gamma^{\rm trp}(s,t))=\left( \min_{\mu \in \mathcal{P}(\Gamma)} (\mu^T \mathcal{N} \mathcal{N}\mu)\right)^{-1} 
$$
\begin{equation}
\label{eq:Mod2Pawgraph}
    = \left( \min_{\mu \in \mathcal{P}(\Gamma)} ( \mu^T \begin{bmatrix}
    2\alpha^2 & \alpha \beta \\
    \alpha \beta & 3\beta^2
    \end{bmatrix} \mu )\right)^{-1}= \left( \min_{\delta \in [0,1]} (  2 \alpha^2 \delta^2 + 2\alpha \beta (1-\delta)\delta + (1-\delta)^2 3 \beta^2\right)^{-1}.
\end{equation}
If $\alpha<\beta/2$ or $\alpha>3\beta$, the optimal value of $\delta$ will occur at one of the two endpoints, implying that the optimal plan uses only one of the two paths.  For $\alpha$ between $\beta/2$ and $3\beta$, the optimal value of $\delta$ is
$$
\delta^* = \frac{\beta(3\beta - \alpha )}{2\alpha^2 - 2\alpha \beta + 3 \beta^2}.
$$
Thus, we find
\begin{equation*}
\Mod_{2,1}(\Gamma^{\trp}(s,t)) =
\begin{cases}
\hspace{3em}(2\alpha^2)^{-1} & \text{if }\alpha < \frac{\beta}{2},\\
\hspace{3em}(3\beta^2)^{-1} & \text{if }3\beta < \alpha,\\
\frac{(2\alpha^2-2\alpha\beta+3\beta^2)^2}
{\alpha^2\beta\left(
15\beta^3 - 16 \alpha\beta^2 + 24\alpha^2\beta - 4\alpha^3
\right)} &\text{ if } \frac{\beta}{2} \le \alpha \le 3\beta.
\end{cases}
\end{equation*}
This example exhibits an interesting behavior.  If either of the two paths is excessively long (in the temporal sense) relative to the other, it is effectively omitted by temporal modulus.
\end{exmp}

\begin{exmp}
\label{ex:Big}
As a final example, we consider the temporal modulus of time respecting paths on the \emph{CollegeMsg} network from the SNAP network collection~\cite{snapnets}.  The original data is a set of tuples $(u,v,t)$ where $u$ and $v$ are users of an online social network and $t$ is a time at which $u$ sent a private message to $v$.  (The time stamps in the data set are Unix epoch seconds.)  We have converted this to a directed graph with an edge $(u,v)$ marking that user $u$ sent user $v$ at least one private message and with $T((u,v))$ containing the set of time stamps of all such messages.  Intuitively, a time respecting path in this graph provides a potential pathway through which information could have flowed from the source user to the target user, so one might think of $\Mod_{2,1}(\Gamma^{\trp}(s,t))$ as a measure of the potential social influence user $s$ may have had on $t$.

From this graph, we extracted the largest weakly connected component, leaving us with a directed temporal graph containing 1,893 vertices and 20,292 edges representing a total of 59,831.  Next, we selected two vertices $v_1$ and $v_2$ (vertices numbered 425 and 501 in the data set) with the property that the two vertices were somewhat distant in the network.  (The shortest time respecting path in either direction requires at least 3 hops.)   We then used the basic algorithm described in~\cite{albin2015modulus} to approximate $\Mod_{2,1}(\Gamma^{\trp}(v_1,v_2))$ and $\Mod_{2,1}(\Gamma^{\trp}(v_2,v_1))$.  For these calculations, we used per-edge penalization.  To introduce a time penalty, we used a normalized exponential of the form
\begin{equation*}
\varphi(t) = \exp\left(10^{-7}(t-t_0)\right),
\end{equation*}
where $t_0$ is earliest timestamp found among all edges leaving the source vertex.  The approximate modulus in these two cases is
\begin{equation*}
\Mod_{2,1}(\Gamma^{\trp}(v_1,v_2)) \approx 1.41,\quad
\Mod_{2,1}(\Gamma^{\trp}(v_2,v_1)) \approx 1.87.
\end{equation*}
The potential for information transfer seems to be slightly biased in favor of information flowing from $v_2$ to $v_1$.

\begin{figure}
    \centering
    \includegraphics[width=\textwidth]{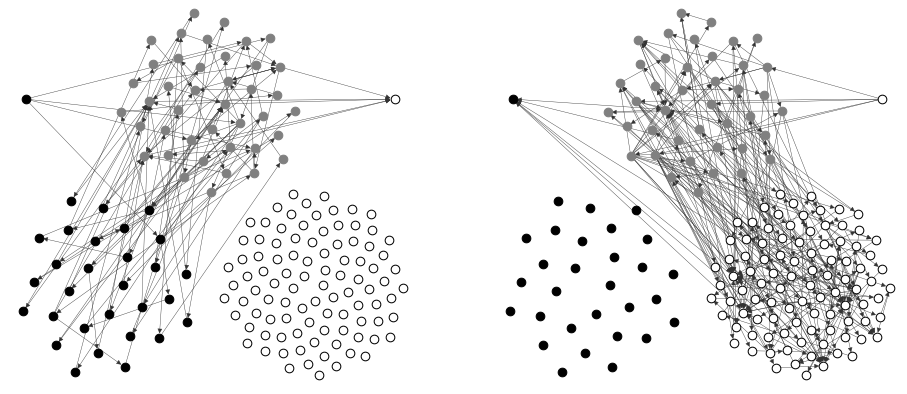}
    \caption{Visualization of email pathways from Example~\ref{ex:Big}.}
    \label{fig:collegemsg}
\end{figure}

Figure~\ref{fig:collegemsg} provides an attempt at visualizing the information provided by modulus.  In order to produce the figures, we first computed the modulus approximations along with approximations to the optimal plans $\mu_1$ and $\mu_2$ (in the mass transport interpretation) for $\Mod_{2,1}(\Gamma^{\trp}(v_1,v_2))$ and $\Mod_{2,1}(\Gamma^{\trp}(v_2,v_1))$ respectively.  Next, we ordered the time respecting paths in $\Gamma^{\trp}(v_1,v_2)$ by decreasing value of $\mu_1$, and similarly for $\Gamma^{\trp}(v_2,v_1)$ and $\mu_2$.  Next, we let $\Gamma_1\subset\Gamma^{\trp}(v_1,v_2)$ be the 33 paths with largest $\mu_1$-mass, and we let $\Gamma_2\subset\Gamma^{\trp}(v_2,v_1)$ be the 88 paths with largest $\mu_2$-mass.  These subsets were chosen so that $\mu_i(\Gamma_i)\approx 0.5$ for $i=1,2$.  In other words, we have chosen the smallest subsets $\Gamma_i$ such that the respective mass transport plan sends at least half of the total mass along these paths.

In both sub-figures of Figure~\ref{fig:collegemsg}, the black vertex on the upper far left represents $v_1$ and the white vertex on the upper far right represents $v_2$.  The remaining black vertices are vertices that participate in paths from $\Gamma_1$ but not in paths from $\Gamma_2$.  Similarly, the remaining white vertices participate in paths from $\Gamma_2$ but not $\Gamma_1$.  The gray vertices participate in paths from both sets.  On the left side of the figure, we have drawn the union of all edges from paths in $\Gamma_1$ and on the right we have drawn the union of all edges from paths in $\Gamma_2$.  It's not possible to observe the temporal information in these plots, there is a significant difference in diversity between the paths of $\Gamma_1$ and the paths of $\Gamma_2$, which reflects the difference in the two modulus values.

\begin{figure}
    \centering
    \includegraphics[width=0.9\textwidth]{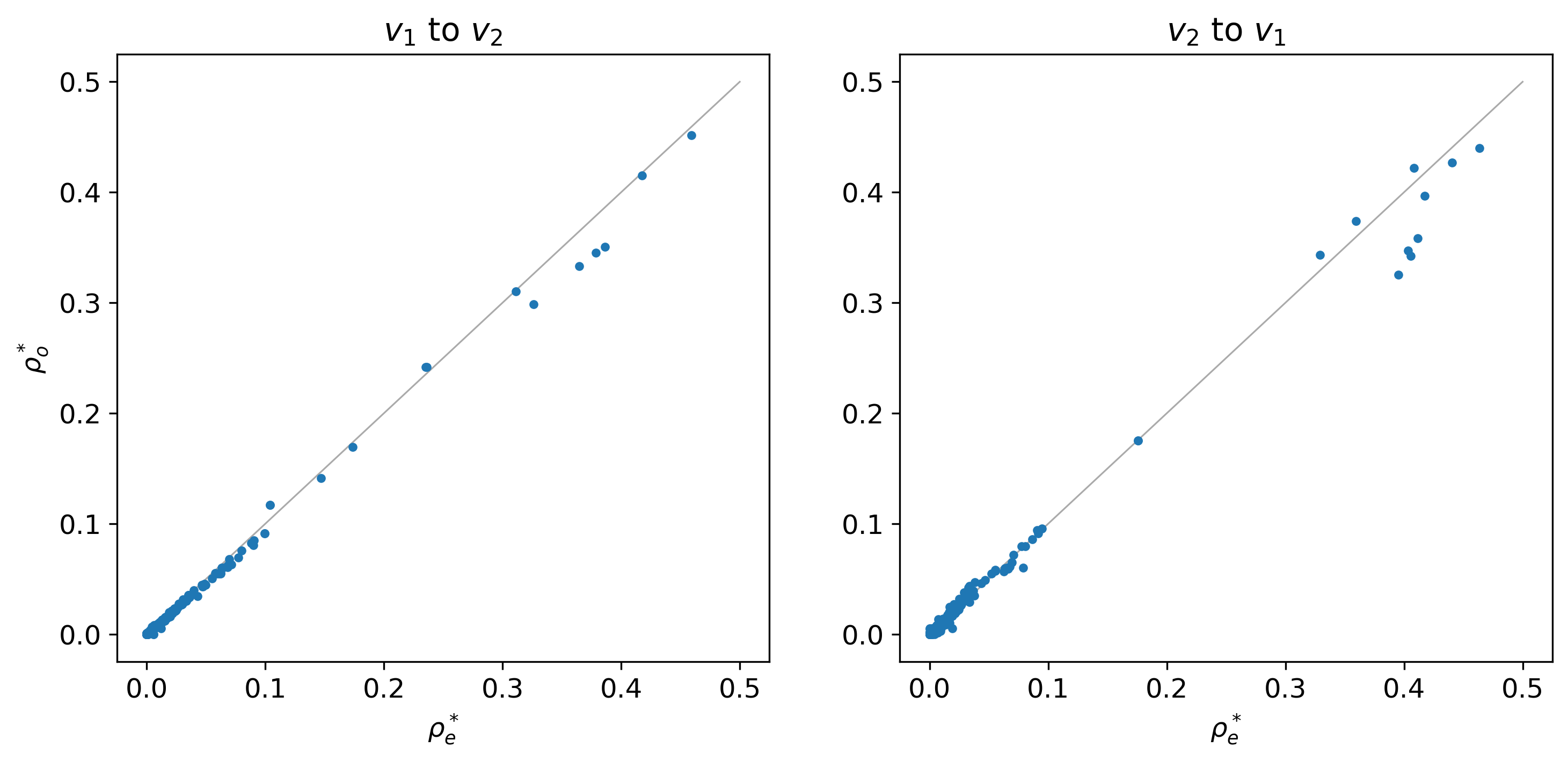}
    \caption{Comparison of the optimal densities for $\Mod_{2,1}(\Gamma^{\trp}(v_1,v_2))$ (left) and $\Mod_{2,1}(\Gamma^{\trp}(v_2,v_1))$ (right) using per-edge or per-object temporal penalization.  Each edge is plotted with $x$ coordinate given by the value of $\rho^*$ using per-edge penalization and $y$ coordinate given by the corresponding value using per-object penalization.}
    \label{fig:rho-star-comp-ex}
\end{figure}

Per-object penalization gives similar results, but with the two path families taking somewhat smaller modulus values:
\begin{equation*}
\Mod_{2,1}(\Gamma^{\trp}(v_1,v_2)) \approx 1.29,\quad
\Mod_{2,1}(\Gamma^{\trp}(v_2,v_1)) \approx 1.67.
\end{equation*}
This difference can be explained by the fact that the usage matrix in the per-edge case takes values that are bounded above by the corresponding entries in the per-object usage matrix, leading to smaller values of $\rho^*$ in the latter case.  Figure~\ref{fig:rho-star-comp-ex} provides a visualization of this fact.  On most (but not all) edges, the optimal density for modulus with per-edge penalization is larger than the corresponding value with per-object penalization.
\end{exmp}

\newpage
\section{Conclusion}

We have defined a version of $p$-modulus for a class of object families on temporal graphs that includes families of time respecting paths.  By transforming the temporal modulus problem into a generalized modulus problem on an aggregate (static) graph, we have shown that much of the established theory of $p$-modulus can be transferred directly to temporal modulus.  Through a set of examples, we have demonstrated that this new concept of modulus is able to capture not only the sizes and diversity of the objects in the family, but also the temporal data included in these objects.

\section*{Acknowledgments}
This material is based upon work supported by the National Science Foundation under Grant No.~1515810.

\bibliographystyle{acm}
\bibliography{references}

\end{document}